\documentclass[10pt]{amsart}

  \baselineskip 1 mm

\usepackage{amsmath,amsthm,amsfonts,latexsym,amsopn,verbatim,amscd,amssymb}
\usepackage{hyperref}

\theoremstyle{plain}
\newtheorem{thm}{Theorem}[section]
\newtheorem{lem}[thm]{Lemma}
\newtheorem{prop}[thm]{Proposition}
\newtheorem{cor}[thm]{Corollary}
\theoremstyle{definition}

\newtheorem{remark}[thm]{Remark}
\newtheorem{Def}[thm]{Definition}

\numberwithin{equation}{section}

\DeclareMathOperator{\Irr}{Irr}

\newcommand{\lin}{\textnormal{lin}}
\newcommand{\nl}{\textnormal{nl}}

\newcommand{\bnum}{\begin{enumerate}}
\newcommand{\enum}{\end{enumerate}}
\newcommand{\equa}[1]{%
\begin{equation*}
\begin{aligned}
#1
\end{aligned}
\end{equation*}
}

\begin{document}

\title{On Generalized Commutator}
\author[S. K. Prajapati and R. K. Nath]{S. K. Prajapati and R. K. Nath*}
\address{Sunil Kumar Prajapati\\
 Einstein Institute of Mathematics, Hebrew University of Jerusalem,
 Jerusalem 91904, Israel}
\email{skprajapati.iitd@gmail.com}
\address{Rajat Kanti Nath\\
Department of Mathematical Sciences, Tezpur
University,  Napaam-784028, Sonitpur, Assam, India.}
\email{rajatkantinath@yahoo.com}

\subjclass[2010]{primary 20D60, 20F12, 20F70,  secondary 20C15}
\keywords{Word map, Commutator, Irreducible character, finite group}

%

\thanks{*Corresponding author}

%
%
\begin{abstract}In this paper, we consider some generalized commutator equations in a finite group and show
that the number of solutions of such equations are characters of that group. We also obtain explicit formula for
this character, considering the equation $[\cdots[[[x_1,x_2],x_3],$ $x_4],\cdots, x_n] = g$, for some well-known classes of finite groups in terms of orders of the group, its center and its commutator subgroup. This paper is an extension of the works of Pournaki and Sobhani in \cite{mP04}.
\end{abstract}

\maketitle

%
%

\section{Introduction} \label{S:intro}
During the last years there was a great interest
in word maps in groups. These maps are defined as follows.
Let
 $F(x_1, x_2, \dots, x_n)$ denote the free group of words  on $n$ generators  $x_1, x_2, \dots, x_n$. Let $w = w(x_1,\cdots , x_n)$ be a non-trivial reduced word in $F(x_1, x_2, \dots, x_n)$. Then we may write $w = x^{m_1}_{i_1}x^{m_2}_{i_2}\cdots x^{m_r}_{i_r}$
where $1 \leq i_j \leq n, m_j \in \mathbb{Z}$. Let $G$ be a finite group. Then the word $w(x_1,x_2,\dots,x_n)$ defines
the map $w:G^n\longrightarrow G$ given by $(g_1,g_2,\dots,g_n)\mapsto w(g_1,g_2,\dots,g_n).$ The map $w:G^n\longrightarrow G$ is called a word map. For $g\in G$ we have
\[
|w^{-1}(g)| = |\{(a_1, a_2, \dots, a_n) \in G^n : w(a_1, a_2, \dots, a_n) = g \}|.
\]
Consider the map
$\zeta^w_G: G\longrightarrow \mathbb{N}\cup \{0\}$ given by $g\mapsto |w^{-1}(g)|.$ Note that $\zeta^w_G(g)$ denotes the number of solutions of the equation $w(x_1, x_2, \dots, x_n) = g$. For any word $w$, the map $\zeta^w_G$ is a class function on $G$ but not necessarily a character. Equations in finite groups have been studied by many authors since the days of Frobenius (see for example \cite{aV11, DN09, pS12, SR15, SR13,  Sol69, sP95} etc.).
If $w(x_1, x_2) = [x_1, x_2] := x_1^{-1}x_2^{-1}x_1x_2$, then Frobenius \cite{fro68} showed that
\begin{equation}\label{Frobeq}
\zeta^w_G = \sum_{\chi\in \Irr(G)}\frac{|G|}{\chi(1)}\chi,
\end{equation}
which is a character of $G$. Considering $w(x_1, x_2, \dots, x_n) = x_1 \cdots x_nx_1^{-1} \cdots x_n^{-1}$ or $[x_1, x_2]\cdots [x_{n-1}, x_n]$ Tambour \cite{tT98} showed that
$\zeta^w_G$ is a character of $G$. Extending the results of Tambour, Das and Nath \cite{DN09} showed that $\zeta^w_G$ is a character of $G$ if $w$ is an admissible word. Further, Parzanchevski and  Schul \cite{pS12} have considered some classes of words and extended the results of Das and Nath \cite{DN09}.

 In this paper, we consider the word $w(x_1, x_2, \dots, x_m) = [w_1(x_1,\cdots, x_n), w_2(x_{n+1},$ $\cdots, x_m)]$ and show that $\zeta^w_G$ is a character of $G$ under some condition. This extends a classical result of Frobenius mentioned above. Note that the word $w(x_1, x_2, \dots, x_m) = [w_1(x_1,\cdots, x_n),$ $w_2(x_{n+1},\cdots, x_m)]$ includes words that are not considered in \cite{pS12}. As a consequence of our result we show that  $\zeta^{w_n}_G$ is a character of $G$ when $w_{n}(x_1,\cdots,x_n)=[\cdots[[[x_1,x_2],x_3],$ $x_4],\cdots, x_n]$.
    The rest part of this paper is devoted to derive explicit formula for $\zeta^{w_n}_G$ for different classes of finite groups such as groups having unique non-linear irreducible character, Camina $p$-groups, generalized Camina groups etc.
It is worth mentioning that the case when $n = 2$ is considered in \cite{DN10, sG07, nY2015, mP04}.

Let $P^{(n)}_{w}$ be the probability distribution associated to the word map $w$ induced by $w(x_1, x_2, \dots, x_n)$ on $G$. Therefore, $P^{(n)}_{w} (g) = \dfrac{\zeta^w_G(g)}{|G|^n}$. In recent years, many authors have studied the probability distribution associated to the word maps induced by different kinds of words (see \cite{abert06, DN11, nath14, nath11, nikolov07} etc.). As an application of our results, one may compute the probability distribution  associated to the word map induced by the word $w_n(x_1, x_2, \dots, x_n)= [[\cdots[[x_1, x_2], x_3]\cdots], x_n]$  for the different families of finite groups considered in Section 3.
It may be mentioned here that ${P}^{(n)}_{w_n} (g)$ when  $g = 1$, the identity element of $G$, is studied by  Moghaddam et al. \cite{mrM05} in the name of $n$-th nilpotency degree of $G$.
 If $n = 2$ then we have ${P}^{(2)}_{w_n}(g) = {P}_g(G)$, a notion introduced and studied by Pournaki and   Sobhani in \cite{mP04}.

The following notations
are standard and essential for this paper. Throughout this paper $G$ denotes
 a finite group, $\Irr(G), \lin(G)$ and $\nl(G)$ denote the set of all irreducible characters,  linear irreducible characters and  nonlinear irreducible characters of $G$ respectively. For any normal subgroup $N$ of $G$, $\Irr(G|N)$ denotes  $\Irr(G)\setminus \Irr(G/N)$. Also, for any character $\chi$ of $G$ we write $\chi \downarrow_N$ to denote  the restriction of $\chi$ on $N$. If $\chi$ is a character of any subgroup $H$ then we write  $\chi \uparrow_H^G$ to denote the induced character of $G$ induced by $\chi$. The inner product of two characters $\phi$ and $\psi$ of $G$ is given by $\langle\phi, \psi\rangle = \frac{1}{|G|}\sum_{g\in G}\phi(g)\overline{\psi(g)}$ and $\langle\phi, \psi\rangle_H = \frac{1}{|H|}\sum_{g\in H}\phi(g)\overline{\psi(g)}$ for any subgroup $H$ of $G$. For any integer $i\geq 1$ define
inductively $Z_0(G)=\{1\}, Z_1(G) = Z(G)$ the center of G, $Z_i(G)/Z_{i-1}(G)=$ the center of $G/Z_{i-1}(G)$. Also $\gamma_1(G) = G$,
$\gamma_2(G) = [G, G] = G{}'$ the commutator subgroup of $G$ and $\gamma_{i+1}(G) = [\gamma_i(G), G]$.


\section{Commutator of two words}
In this section, we consider  the word
$$w(x_1,\dots,x_m):=[w_1(x_1,\cdots, x_n), w_2(x_{n+1},\cdots, x_m)]$$
and obtain some properties of   $\zeta^w_G$.
A word $w$ is called measure preserving with respect to a finite group $G$ if all the fibers
of this map are of equal size, i.e. $\zeta^w_G$ is a constant map.
We have the following main result of this section.
\begin{thm}\label{thmtype1}
Let $H$ be  a normal subgroup of $G$ and
$$
w(x_1,\dots,x_m):=[w_1(x_1,\cdots, x_n), w_2(x_{n+1},\cdots, x_m)]
$$
such that $m-n-1\geq 0$ and  the word map $w_2$ is measure preserving   with respect to $G$.
Then we have the following:
\begin{enumerate}
\item  If $m - n = 1$ and $\zeta^{w_1}_H$ is a character,  then $\zeta^w_G$ is a character constant on $G$-conjugacy classes of $H$.
\item  If $m - n \geq 2$ and $\zeta^{w_1}_H$ is a character, then $\zeta^w_G$ is a character.
\end{enumerate}
\end{thm}
\begin{proof}
 Suppose $\chi\in \Irr(G)$ is afforded by the irreducible representation $\rho_\chi$ of $G$.
Consider the element
$$
 z=\sum_{\substack{x_1,\dots,x_n\in H\\x_{n+1},\dots,x_m\in G}} [w_1(x_1,\cdots, x_n), w_2(x_{n+1},\cdots, x_m)].
 $$

\noindent Then, we have
\begin{eqnarray}\label{s}\displaystyle
\rho_\chi(z)&=& \sum_{x_1,\dots,x_n\in H}\rho_\chi(w_1(x_1,\cdots ,x_n)^{-1}) A_\chi(x_1,\cdots, x_n)
\end{eqnarray}

\noindent where
\begin{align}\label{lb}\displaystyle
A_\chi(x_1,&\cdots ,x_n) = \\ \nonumber
&\sum_{x_{n+1},\dots,x_m\in G}\rho_\chi(w_2(x_{n+1},\cdots, x_m)^{-1}w_1(x_1,\cdots, x_n)w_2(x_{n+1},\cdots, x_m)).
\end{align}
Since the word  $w_2$ is measure preserving with respect to $G$, $\zeta^{w_2}_G$ is a constant function and hence
$A_\chi(x_1,\cdots ,x_n)$ commutes with $\rho_\chi(g)$  for all $ g\in G$.  Hence, by Schur's  lemma (see \cite[Lemma 2.25]{Isaacs}), we have
\begin{equation}
\label{la}A_\chi(x_1,\cdots, x_n)= \alpha I
\end{equation}
 for some $\alpha \in \mathbb{C}$. Now
taking trace on the both sides of (\ref{la}) and using (\ref{lb}) we get
\begin{eqnarray}\label{lc}
\alpha &=& \frac{|G|^{m-n}}{\chi(1)} \chi(w_1(x_1,\cdots , x_n)).
 \end{eqnarray}
 Again taking trace on both side of \eqref{s} and using (\ref{la}),(\ref{lc}) we get
\begin{eqnarray}\label{l4}
\chi(z) &=& \displaystyle \frac{|G|^{m-n}}{\chi(1)} \sum_{h\in H}\zeta^{w_1}_H(h) \chi(h)\chi(h^{-1})\nonumber\\
                           &=& \frac{|G|^{m-n}|H|}{\chi(1)} \langle \zeta^{w_1}_H \overline{\chi},\overline{\chi} \rangle_H.
\end{eqnarray}
\noindent By definition of $\zeta^{w}_G$, we have
\begin{eqnarray}\label{l3}
\chi(z) &=& |G| \langle \zeta^{w}_G, \overline{\chi} \rangle.
\end{eqnarray}
\noindent Therefore, by  (\ref{l4}) and (\ref{l3}), we have
\begin{equation*}
 \zeta^{w}_G = \displaystyle \sum_{\chi\in \Irr(G)}\frac{|G|^{m-n-1}|H|}{\chi(1)} \langle \zeta^{w_1}_H \overline{\chi},\overline{\chi} \rangle_H ~\chi.
\end{equation*}

\noindent Since $\langle \zeta^{w_1}_H \overline{\chi},\overline{\chi} \rangle_H = \langle \zeta^{w_1}_H \chi,\chi \rangle_H$,
 from the above equation we get
\begin{equation}\label{l5}
 \zeta^{w}_G = \displaystyle \sum_{\chi\in \Irr(G)}\frac{|G|^{m-n-1}|H|}{\chi(1)} \langle \zeta^{w_1}_H \chi,\chi \rangle_H ~\chi.
\end{equation}

(1) It is enough to prove that $\frac{|H|}{\chi(1)} \langle \zeta^{w_1}_H \chi,\chi \rangle_H$ is an integer.
 Since $H$ is a normal subgroup of $G$, $H=\displaystyle \bigcup_{1\leq j \leq t}Cl_G(x_j)$ for some $x_j\in H$. For any irreducible character $\chi$ of $G$, we have
 \begin{eqnarray*}\label{A3}
 |H|~ \langle \zeta^{w_1}_H \chi,\chi \rangle_H
 &=& \sum_{h\in H}\zeta^{w_1}_H(h)\chi(h)\chi(h^{-1})\\
 &=& \sum_{1\leq j \leq t}|Cl_G(x_j)|~\zeta^{w_1}_H(x_j)~\chi(x_j)~\chi(x_j^{-1}).
 \end{eqnarray*}
 Suppose
 $K_j= \displaystyle \sum_{x\in Cl_G(x_j)}x$, for $1\leq j \leq t$. Then
 $w_{\chi}(K_j)= \frac{\chi(x_j)|Cl_G(x_j)|}{\chi(1)}$ is an algebraic integer for any irreducible character.
 Hence, we have
$$
\frac{\mid H\mid}{\chi(1)} \langle \zeta^{w_1}_H \chi, \chi \rangle_H= \sum_{1\leq j \leq k}\zeta^{w_1}_H(x_j)w_{\chi}(K_j)\chi(x_j^{-1}).
$$
The right hand side of the above equality is an algebraic integer. Thus left hand side is an
algebraic rational number and hence the result follows.

(2) It is clear that the coefficient
 $\frac{|G|^{m-n-1}|H|}{\chi(1)} \langle \zeta^{w_1}_H \chi,\chi \rangle_H$ in \eqref{l5} is a rational integer. Hence $\zeta^{w}_G$ is a character of $G$.
 \end{proof}

\begin{cor}
Let $H$ be a  normal subgroup of $G$ and
$$
w(x_1,\dots,x_m):=[x_1, w_2(x_2,\cdots, x_n)]
$$
where $x_1\in H$   and $n - 2 \geq 0$. If  $w_2$ is measure preserving  with respect to $G$ then $\zeta^w_G$ is a character of $G$.
\end{cor}
\begin{proof}
In  view of the proof of Theorem \ref{thmtype1}, we have
\[
\zeta^{w}_G = \sum_{\chi\in \Irr(G)}\frac{|G|^{n-2}|H|}{\chi(1)} \langle \chi,\chi \rangle _H~\chi.
 \]
If $n - 2 > 0$ then   it is clear that
 $\frac{|G|^{n-2}.|H|}{\chi(1)} \langle \chi,\chi\rangle_H$   is a rational integer for all $\chi\in \Irr(G)$. Hence, $\zeta^w_G$ is a character of $G$. If $n - 2 = 0$ then, using the same argument as in the proof of  Theorem \ref{thmtype1} part (1), one can complete the proof.
\end{proof}

In the above Theorem \ref{thmtype1}, taking $H = G$  we get the following result.
\begin{cor}\label{cor:commutator}
Let
$w(x_1,\dots,x_m):=[w_1(x_1,\cdots, x_n), w_2(x_{n+1},\cdots, x_m)].$
If $\zeta^{w_1}_G$ is a character and  the word $w_2$ is measure preserving  with respect to $G$, then $\zeta^w_G$ is a character of $G$ and
\[
\zeta^{w}_G  =  \sum_{\chi\in \Irr(G)}\frac{|G|^{m-n}}{\chi(1)} \langle \zeta^{w_1}_G \chi,\chi \rangle ~\chi.
\]
\end{cor}

\section{Explicit formula for $\zeta^{w_n}_G$}

In this section, we consider the word $w_{n}(x_1,\cdots,x_n) := [\cdots[[[x_1,x_2],x_3],x_4],\cdots, x_n]$ for $n\geq 2$. Then  $\zeta^{w_n}_G(g)$ denotes the number of solutions of the equation
\[
[\cdots[[[x_1,x_2],x_3],x_4],\cdots, x_n] = g.
\]
Note that the word $w_n$ can be defined recursively as
\equa{w_2(x_1,x_2)&=[x_1,x_2]=x_1^{-1}x_2^{-1}x_1x_2,\\
w_{3}(x_1,x_2,x_3)&=[w_2(x_1,x_2),x_3],\\
&\,\,\,\vdots\\
w_{n}(x_1,\cdots,x_n)&=[w_{n-1}(x_1, \dots, x_{n -1}),x_n].}
Thus $w_n$ is a particular type of the words considered in Section 2.
By Corollary \ref{cor:commutator}, it follows that $\zeta^{w_n}_G$ is a character of $G$ and
\begin{eqnarray}\label{gencommutsolfunction}
\zeta^{w_n}_G&=&|G|\sum_{\chi\in \Irr(G)}\frac{C^{w_{n}}(\chi)\chi}{\chi(1)}
\end{eqnarray}
where $C^{w_{n}}(\chi)=\langle \zeta^{w_{n-1}}_G \chi,\chi \rangle$. Note that $ C^{w_{m+1}}(\chi) = \chi(1)^2|G|^{m-1}$ whenever $m$ is greater than the nilpotency class of $G$.
For $n=2$, $C^{w_{2}}(\chi)=1$ and so equation \eqref{gencommutsolfunction} reduces to \eqref{Frobeq}; which is a classical result of Frobenius. For $n\geq 3$ there is no such formula for $C^{w_{n}}(\chi)$. However, the following proposition shows that $C^{w_n}(\chi)= |G|^{n - 2}$ if $\chi \in \lin(G)$.

\begin{prop}\label{GCPcoffeicientofirr_prop1}
 Let $G$ be a finite group and  $w_{n}(x_1,\cdots,x_n)$ denote the word $[\cdots[[x_1,x_2],x_3],\cdots, x_n]$ with $n\geq 3$. Then we have following.
\begin{enumerate}
\item If $\chi\in \lin(G)$, then $C^{w_n}(\chi)=\langle \zeta^{w_{n-1}}_G, 1_G\rangle$.
\item $\langle \zeta^{w_{n-1}}_G, 1_G\rangle=|G|\langle \zeta^{w_{n-2}}_G, 1_G\rangle$. In particular, $\langle \zeta^{w_{n-1}}_G, 1_G\rangle=|G|^{n-2}$.
\end{enumerate}
\end{prop}
\begin{proof} $(1)$ Let $\chi\in \lin(G)$. Then
\[
C^{w_n}(\chi)=\langle \zeta^{w_{n-1}}_G\chi, \chi\rangle=\frac{1}{|G|}\sum_{g\in G}\zeta^{w_{n-1}}_G(g)|\chi(g)|^2=\langle \zeta^{w_{n-1}}_G, 1_G\rangle.
\]
$(2)$ Using \eqref{gencommutsolfunction}, we have
\begin{eqnarray*}
\langle \zeta^{w_{n-1}}_G, 1_G\rangle &=& \frac{1}{|G|}\sum_{g\in G}\zeta^{w_{n-1}}_G(g)\\
  &=& \frac{1}{|G|}\sum_{\chi\in \Irr(G)}\bigg( \sum_{g\in G}\frac{|G|}{\chi(1)}\langle \zeta^{w_{n-2}}_G\chi,\chi\rangle \chi (g)\bigg)\\
 &=& \frac{1}{|G|}\sum_{\chi\in \Irr(G)}\bigg(\frac{|G|}{\chi(1)}\langle \zeta^{w_{n-2}}_G\chi,\chi\rangle |G|\langle \chi, 1_G\rangle\bigg)\\
  &=&|G| \langle \zeta^{w_{n-2}}_G, 1_G\rangle.
\end{eqnarray*}
Therefore, $\langle \zeta^{w_{n-1}}_G, 1_G\rangle=|G|^{n-2}$.
\end{proof}

In view of the above proposition, \eqref{gencommutsolfunction} becomes
\begin{equation}\label{neweq0001}
\zeta^{w_n}_G = |G|^{n - 1}\sum_{\chi\in \lin(G)}\chi +  |G|\sum_{\chi\in \nl(G)}\frac{C^{w_{n}}(\chi) ~\chi}{\chi(1)}.
\end{equation}
In the following subsections we compute   $C^{w_{n}}(\chi)$ where $\chi \in \nl(G)$  for $n\geq 3$ and hence derive explicit formula for $\zeta^{w_n}_G(g)$  for some classes of finite  groups.

\subsection{$n$-isoclinism and Camina group}
The notion of $n$-isoclinism of groups is implicit in a short note of P.
Hall \cite{Hall} on verbal and marginal subgroups.
\begin{Def}\label{def:n-isoclinic}
Two groups $G$ and $H$ are $n$-isoclinic, if there exist isomorphisms $\phi$ and $\psi$:
\begin{enumerate}
\item $\phi$ is an isomorphism from $G/Z_n(G)$ to $H/Z_n(H)$;
\item $\psi$ is an isomorphism from $\gamma_{n+1}(G)$ to $\gamma_{n+1}(H)$;
\item $\psi$ is compatible with $\phi$, i.e. $\psi([\cdots[[g_1,g_2],g_3],\cdots, g_{n+1}])=[\cdots[[h_1,h_2],h_3],\cdots, h_{n+1}]$,
for any $h_i\in \phi(g_iZ_n(G))$, $g_i\in G$, $1\leq i\leq n+1$. In other words, we have a commutative diagram.
\end{enumerate}
\begin{center}
$\begin{CD}
\frac{G}{Z_n(G)}\times \cdots \times \frac{G}{Z_n(G)}  @>\phi\times \cdots\times \phi>>  \frac{H}{Z_n(H)}\times \cdots \times \frac{H}{Z_n(H)}  \\
@VV a_{(n+1,G)}V   @VV a_{(n+1,H)}V \\
\gamma_{n+1}(G)  @>\psi>>   \gamma_{n+1}(H),
\end{CD}$
\end{center}
where
$$a_{(n+1,G)}\big(\big(g_1Z_n(G), \cdots, g_{n+1}Z_n(G)\big)\big)=[\cdots[g_1,g_2],\cdots, g_{n+1}]$$
and $$a_{(n+1,H)}\big(\big(h_1Z_n(H), \cdots, h_{n+1}Z_n(H)\big)\big)=[\cdots[h_1,h_2],\cdots, h_{n+1}].$$
\end{Def}

\noindent We say that the pair $(\phi, \psi)$ is an $n$-isoclinism between $G$ and $H$. It is clear from the definition that an $n$-isoclinism induces also an $(n+1)$-isoclinism.
 If $n=1$, then $G$ and $H$ are called isoclinic groups.

In the following theorem we generalize  \cite[Lemma 3.5]{mP04}.
\begin{thm}\label{thm:isoclinic}
Let $G$ and $H$ be two finite groups and let $(\phi, \psi)$ be an isoclinism from $G$ to $H$.
If $g\in \gamma_{n+1}(G)$ then $\zeta^{w_{n+1}}_G(g) = (|G|/|H|)^{n+1}\zeta^{w_{n+1}}_H(\phi(g))$.
\end{thm}
\begin{proof} Let the pair $(\phi, \psi)$ be an $n$-isoclinism between $G$ and $H$. Consider the group $\overline{G}=G/Z_n(G)$ and $\overline{H}=H/Z_n(H)$. Then
\begin{align*}
\frac{1}{|Z_n(G)|^{n+1}}&\zeta^{w_{n+1}}_G(g)\\
&=\frac{1}{|Z_n(G)|^{n+1}}|\{(g_1,g_2,\cdots,g_{n+1})\in G^{n+1} ~|~ [\cdots[[g_1,g_2],g_3],\cdots, g_{n+1}]=g \}|\\
&=\frac{1}{|Z_n(G)|^{n+1}}|\{(g_1,g_2,\cdots,g_{n+1})\in G^{n+1} ~|~ a_{(n+1,G)}\big([\cdots[[\bar{g}_1,\bar{g}_2],\bar{g}_3],\cdots, \bar{g}_{n+1}]\big)=g \}|\\
&=|\{(\bar{g}_1,\bar{g}_2,\cdots,\bar{g}_{n+1})\in \overline{G}^{n+1} ~|~ a_{(n+1,G)}\big([\cdots[[\bar{g}_1,\bar{g}_2],\bar{g}_3],\cdots, \bar{g}_{n+1}]\big)=g \}|\\
&=|\{(\bar{g}_1,\bar{g}_2,\cdots,\bar{g}_{n+1})\in \overline{G}^{n+1} ~|~ \psi \big(a_{(n+1,G)}\big([\cdots[[\bar{g}_1,\bar{g}_2],\bar{g}_3],\cdots, \bar{g}_{n+1}]\big)\big)=\psi(g) \}|\\
&(\textnormal{since $\psi$ is an isomorphism})
\end{align*}
\begin{align*}
&=|\{(\bar{g}_1,\bar{g}_2,\cdots,\bar{g}_{n+1})\in \overline{G}^{n+1} ~|~ a_{(n+1,H)}\big([\cdots[[\phi(\bar{g}_1),\phi(\bar{g}_2)],\phi(\bar{g}_3)],\cdots, \phi(\bar{g}_{n+1})]\big)=\psi(g) \}|\\
&(\textnormal{by commutativity of diagram as in Definition \ref{def:n-isoclinic} })\\
&=|\{(\bar{h}_1,\bar{h}_2,\cdots,\bar{h}_{n+1})\in \overline{H}^{n+1} ~|~ a_{(n+1,H)}\big([\cdots[[\bar{h}_1,\bar{h}_2],\bar{h}_3],\cdots, \bar{h}_{n+1}]\big)=\psi(g) \}|\\
&(\textnormal{since $\phi$ is an isomorphism})\\
&=\frac{1}{|Z_n(H)|^{n+1}}|\{(h_1,h_2,\cdots,h_{n+1})\in H^{n+1} ~|~ a_{(n+1,H)}\big([\cdots[[\bar{h}_1,\bar{h}_2],\bar{h}_3],\cdots, \bar{h}_{n+1}]\big)=\psi(g) \}|\\
&=\frac{1}{|Z_n(H)|^{n+1}}|\{(h_1,h_2,\cdots,h_{n+1})\in H^{n+1} ~|~ [\cdots[[h_1,h_2],h_3],\cdots, h_{n+1}]=\psi(g) \}|\\
&=\frac{1}{|Z_n(H)|^{n+1}}\zeta^{w_{n+1}}_H(\phi(g)).
\end{align*}
Hence the result follows.
\end{proof}

The rest part of this subsection is devoted to finite Camina $p$-group. We start with the following definitions.\\
A pair $(G,H)$ is said to be a {\it Camina pair} if $1 < H < G$ is a
normal subgroup of $G$ and for every element $g\in G \setminus H$,
$gH\subseteq Cl_G(g)$. It is easy to see
that if $(G,H)$ is a Camina pair then $Z(G)\leq H \leq G{}'$. A group $G$ is called a \textit{Camina group} if $(G, G')$ is a Camina pair.
It is easy to see that if G is a Camina group then $\chi(G\setminus G{}')=0$ for every $\chi\in \nl(G)$.  A group $G$ is said to be a \textit{Generalized Camina group} if
$Cl_G(g)=gG{}'$ for all $g\in G\setminus G{}'Z(G)$.
Lewis \cite{Lewis2} showed that a finite group $G$ is a generalized Camina group if and only if $G$ is
isoclinic to a Camina group. It follows that if $G$ is a finite  nilpotent  generalized Camina group, then $G$ is isoclinic to a Camina $p$-group for some prime $p$. So, in view of Theorem \ref{thm:isoclinic}, it is enough to compute $\zeta^{w_{n}}_G(g)$ for a Camina $p$-group.
If $G$ is a finite Camina $p$-group then it was shown in \cite{Dark} that
the nilpotency class of $G$ is at most $3$.

A pair of groups $(G, N)$ where $N$ is a normal subgroup of the group $G$ is called a \textit{generalized Camina pair}
(abbreviated as GCP) if $\chi(g)=0$ for all  $\chi\in \nl(G)$ and for all $g \in G \setminus N$. Equivalent definitions of a generalized Camina pair  can be found in \cite{Lewis4}.
 The groups  with $(G,Z(G))$ a GCP were studied under the name $VZ$-groups by Lewis in \cite{Lewis1}.
If $G$ is a finite Camina $p$-group of nilpotency class $2$, then $(G,Z(G))$ is a generalized Camina pair and in this case $\zeta^{w_n}_G(g)$  can be obtained by the following theorem.
\begin{thm}\label{prop_class2}
 Let $G$ be a finite group such that  $(G,Z(G))$ is a GCP.  If $g \in G'$   then
\[
\zeta^{w_{2}}_G(g) = \begin{cases}
\frac{|G|^2}{|G'|}\left(1 - \frac{1}{|G : Z(G)|}\right) &\text{ if } g \neq 1 \\
\frac{|G|^2}{|G'|}\left(1 + \frac{|G'| - 1}{|G : Z(G)|}\right) &\text{ if } g = 1.
\end{cases}
\]
and
\[
\zeta^{w_{n}}_G(g) = \begin{cases}
0 &\text{ if } g \neq 1 \\
|G|^n &\text{ if } g = 1
\end{cases} \text{ if } n \geq 3.
\]
\end{thm}
\begin{proof}
If $(G,Z(G))$ is a GCP then by \cite[Lemma 2.4]{Lewis4} we have that $G'$ is a subgroup of $Z(G)$. That is nilpotency class of $G$ is $2$. Hence the case when $n \geq 3$ follows.

The case when $n = 2$ follows from  \cite[Corollary 2.3]{mP04} and the fact that $cd(G)=\{1, |G:Z(G)|^{1/2} \}$ (which follows from \cite[Section 3]{SR}).
\end{proof}


Now we discuss about finite Camina $p$-group of nilpotancy class $3$. We begin with the following remark which follows from \cite[Section 3]{sunil3}.
\begin{remark}\label{remark:caminapgroup3}
 Let $G$ be a finite Camina $p$-group of nilpotancy class $3$. Then
\begin{enumerate}
\item $G/G{}'$, $G{}'/Z(G)$ and $\gamma_3(G) = [G, G'] = Z(G)$ are elementary abelian $p$-group with $|G : G{}'|=p^{2m}$, $|G{}' : Z(G)|=p^m$ and $|G : Z(G)|=p^{3m}$,
where $m$ is even.
\item $\nl(G)=\Irr(G|Z(G)) \sqcup ~ \nl(G/Z(G))$,
$cd(G)=\{1,p^m,p^{3m/2}\}$, $|\Irr(G|Z(G))|=|Z(G)|-1$ and $|\nl(G/Z(G))|=\frac{|G{}'|}{|Z(G)|}-1$.
\item If $\chi\in \Irr(G|Z(G))$ then $\chi(G\setminus Z(G))=0$ and $\chi{\downarrow}_{Z(G)}=p^{3m/2}\lambda$, where $\lambda\in \Irr(Z(G))$.
\item If $\chi\in \nl(G/Z(G))$ then $\chi(G\setminus G{}')=0$ and $\chi(g)=p^{m}(\lambda\circ \eta)(g)$, where $g\in G{}', \, \lambda\in \Irr(G{}'/Z(G))$
and $\eta:G\longrightarrow G/\gamma_3(G)$ is the natural homomorphism.
\end{enumerate}
\end{remark}
Now, using Remark \ref{remark:caminapgroup3} Part (3) and (4), we get the following result.
\begin{prop}\label{lemma:camina3}
Let $G$ be a Camina $p$-group of nilpotancy class $3$. Then
\begin{eqnarray*}\label{linearsum}
C^{w_3}(\chi) = \left\{
           \begin{array}{ll}
            \langle \zeta^{w_{2}}_G\downarrow_{Z(G)}, 1_{Z(G)}\rangle & \hbox{if $\chi\in \Irr(G|Z(G))$,} \\
            \langle \zeta^{w_{2}}_G\downarrow_{G{}'}, 1_{G{}'}\rangle & \hbox{if $\chi\in \nl(G/Z(G))$.}
           \end{array}
         \right.
\end{eqnarray*}
\end{prop}
\begin{thm}\label{Camina-p_class3}
Let $G$ be a finite Camina $p$-group of nilpotancy class $3$. Then
\begin{eqnarray*}
\begin{aligned}
\zeta^{w_3}_G {} =& |G|^{2}\sum_{\chi\in \lin(G)} \chi+
\bigg(\frac{|G|^{3}}{|G{}'|}+\frac{|G|^2(|G'| - |Z(G)|)}{|Z(G)|}\bigg)\sum_{\chi\in \Irr(G|Z(G))}\frac{\chi}{\chi(1)}\\
&+ \frac{|G|^{3}}{|G{}'|}\sum_{\chi\in \nl(G/Z(G))}\frac{\chi}{\chi(1)}.
\end{aligned}
\end{eqnarray*}
\end{thm}

\begin{proof}
By Remark \ref{remark:caminapgroup3}  part (2) we have  $\nl(G)=\Irr(G|Z(G)) \sqcup ~ \nl(G/Z(G))$. Therefore, using \eqref{neweq0001}, we have
\begin{eqnarray*}
\begin{aligned}
\zeta^{w_3}_G {}= |G|^2\sum_{\chi\in \lin(G)} ~\chi+ |G|\sum_{\chi\in \Irr(G|Z(G))}\frac{C^{w_3}(\chi) ~\chi}{\chi(1)}
+|G|\sum_{\chi\in \nl(G/Z(G))}\frac{C^{w_3}(\chi) ~\chi}{\chi(1)}.
\end{aligned}
\end{eqnarray*}
By  Proposition \ref{lemma:camina3} the above expression becomes
\begin{eqnarray}\label{eqn:camina3fn1}
\begin{aligned}
\zeta^{w_3}_G {} =& |G|^{2}\sum_{\chi\in \lin(G)} \chi+
\langle \zeta^{w_{2}}_G\downarrow_{Z(G)}, 1_{Z(G)}\rangle |G| \sum_{\chi\in \Irr(G|Z(G))}\frac{\chi}{\chi(1)}  \\
&+\langle \zeta^{w_{2}}_G\downarrow_{G{}'}, 1_{G{}'}\rangle |G| \sum_{\chi\in \nl(G/Z(G))}\frac{\chi}{\chi(1)}.
\end{aligned}
\end{eqnarray}
Now,  using Remark \ref{remark:caminapgroup3} part (3) and  (4), and the the expression of $\zeta^{w_{2}}_G$, we get

\begin{eqnarray*}\label{eqn:coffeicientcenter}
\begin{aligned}
\langle \zeta^{w_2}_G\downarrow_{Z(G)}, 1_{Z(G)}\rangle {} =& |G|\langle \varphi\downarrow_{Z(G)}, 1_{Z(G)}\rangle
+|G| \sum_{\lambda\in \Irr(Z(G))\setminus \{1_{Z(G)}\}} \langle \lambda, 1_{Z(G)}\rangle \\
&+|G||\nl(G/Z(G))|.
\end{aligned}
\end{eqnarray*}
where $\varphi = \sum_{\chi\in \lin(G)} \chi$. Since $Z(G) \subseteq G'$ and $\varphi$ is a regular character of $G/G'$, $\langle \varphi\downarrow_{Z(G)}, 1_{Z(G)}\rangle = |G|/|G'|$. Also, $\langle \lambda, 1_{Z(G)}\rangle = 0$ for any $ 1_{Z(G)} \ne \lambda \in \Irr(Z(G))$. Hence,
\[
\langle \zeta^{w_2}_G\downarrow_{Z(G)}, 1_{Z(G)}\rangle {} = \frac{|G|^2}{|G{}'|}+|G||\nl(G/Z(G))| = \frac{|G|^2}{|G{}'|}+ \frac{|G|(|G'| - |Z(G)|)}{|Z(G)|}.
\]
We have
\begin{eqnarray*}
\begin{aligned}
\zeta^{w_2}_G\downarrow_{G{}'} {} =& |G|\sum_{\chi\in \lin(G)}\chi{\downarrow}_{G{}'}+
|G|\sum_{\chi\in \Irr(G|Z(G))} \frac{\chi\downarrow_{G{}'}}{\chi(1)}
+|G|\sum_{\chi\in \nl(G/Z(G))}\frac{\chi\downarrow_{G{}'}}{\chi(1)}.
\end{aligned}
\end{eqnarray*}
 It is easy to see that $\langle \chi{\downarrow}_{G{}'}, 1_{G{}'}\rangle= \langle \phi{\downarrow}_{G{}'}, 1_{G{}'}\rangle = 0$ for
 $\chi\in \Irr(G|Z(G))$ and $\phi\in \nl(G/Z(G))$. Therefore,
$\langle \zeta^{w_2}_G\downarrow_{G{}'}, 1_{G{}'}\rangle=|G|^2/|G{}'|$.
Now putting the values of $\langle \zeta^{w_2}_G\!\!\!\downarrow_{Z(G)}, 1_{Z(G)}\rangle$ and $\langle \zeta^{w_2}_G\downarrow_{G{}'}, 1_{G{}'}\rangle$  in  \eqref{eqn:camina3fn1}  we  get the required expression for $\zeta^{w_3}_G$.
\end{proof}
\noindent The following lemma is useful in proving the next result.
\begin{lem}\label{new_camina_p-gp_1}
Let $G$ be a finite Camina p-group of nilpotancy class $3$. Then
\[
\sum_{\chi\in \Irr(G|Z(G))}\frac{\chi(g)}{\chi(1)} = \begin{cases}
-1 &\text{ if } 1 \neq g \in \gamma_3(G)\\
\;\;\, 0  &\text{ if } g \in G' \setminus \gamma_3(G)
\end{cases}
\]
and
\[
\sum_{\chi\in \nl(G/Z(G))}\frac{\chi(g)}{\chi(1)} = \begin{cases}
\frac{|G'| - |Z(G)|}{|Z(G)|} &\text{ if } 1 \neq g \in \gamma_3(G)\\
 -1  &\text{ if } g \in G' \setminus \gamma_3(G).
\end{cases}
\]
\end{lem}
\begin{proof}
Follows from Remark \ref{remark:caminapgroup3} part (3) and (4).
\end{proof}
\noindent We conclude this subsection with the following result.
\begin{thm}
Let $G$ be a finite Camina p-group of nilpotancy class $3$. If $g \in G'$ then
\[
\zeta^{w_{2}}_G(g) = \begin{cases}
\frac{|G|(|G| - |G'|)}{|G'|} + \frac{|G|(|G'| - |Z(G)|)}{|Z(G)|} &\text{ if } 1 \neq g \in \gamma_3(G) \\
\frac{|G|(|G| - |G'|)}{|G'|}  &\text{ if } g \in G' \setminus \gamma_3(G)\\
\frac{|G|^2}{|G'|} +\frac{|G||G'|}{|Z(G)|} + |G|(|Z(G)| - 2) &\text{ if } g = 1
\end{cases}
\]

and
\[
\zeta^{w_{3}}_G(g) = \begin{cases}
 \frac{|G|^2(|G'| - |Z(G)|)(|G| - |G'|)}{|G'||Z(G)|}  & \text{ if } 1 \neq g \in \gamma_3(G) \\
 0 \hspace{6.3cm} & \text{ if } g \in G' \setminus \gamma_3(G) \\
\frac{|G|^3(|Z(G)|^2 + |G'|  - 1)}{|Z(G)||G'|} + \frac{|G|^2(|G'| - |Z(G)|)(|Z(G)| - 1)}{|Z(G)|} &\text{ if } g = 1.
\end{cases}
\]
\end{thm}

\begin{proof}
Let $n = 2$. If $g = 1$ then \eqref{Frobeq} gives $\zeta^{w_{n}}_G(1) = |G||\Irr(G)|$. Hence, the result follows from Remark \ref{remark:caminapgroup3} part (2). The case when $g \ne 1$ follows from   \cite[Proposition 5.4]{nY2015} noting that $\gamma_3(G) = Z(G)$.

Let  $n=3$.  If $g=1$ then by Theorem \ref{Camina-p_class3} we have
\begin{eqnarray*}
\begin{aligned}
\zeta^{w_n}_G(1) {} =& |G|^{2}\frac{|G|}{|G{}'|}+
\bigg(\frac{|G|^{3}}{|G{}'|}+\frac{|G|^2(|G'| - |Z(G)|)}{|Z(G)|}\bigg)|\Irr(G|Z(G))|
+ \frac{|G|^{3}}{|G{}'|}|\nl(G/Z(G))|\\
=& \frac{|G|^3(|Z(G)|^2 + |G'|  - 1)}{|Z(G)||G'|} + \frac{|G|^2(|G'| - |Z(G)|)(|Z(G)| - 1)}{|Z(G)|},
\end{aligned}
\end{eqnarray*}
using Remark \ref{remark:caminapgroup3} part (2). If $1\neq g\in G'$ then  $\sum_{\chi\in \lin(G)} \chi(g) = |G|/|G'|$. Therefore, by Theorem \ref{Camina-p_class3}, we have
\begin{equation}\label{new_eq_camina_p_1}
\zeta^{w_n}_G(g) {} = \frac{|G|^{3}}{|G{}'|} +
\bigg(\frac{|G|^{3}}{|G{}'|}+\frac{|G|^2(|G'| - |Z(G)|)}{|Z(G)|}\bigg)\sum_{\chi\in \Irr(G|Z(G))}\frac{\chi(g)}{\chi(1)}
+ \frac{|G|^{3}}{|G{}'|}\sum_{\chi\in \nl(G/Z(G))}\frac{\chi(g)}{\chi(1)}
\end{equation}
If $1\neq g\in \gamma_3(G) = Z(G)$ then by  \eqref{new_eq_camina_p_1} and Lemma \ref{new_camina_p-gp_1}, we have
\begin{eqnarray*}
\begin{aligned}
\zeta^{w_n}_G(g) {}
=&\frac{|G|^3}{|G{}'|}+
\bigg(\frac{|G|^{3}}{|G{}'|}+\frac{|G|^2(|G'| - |Z(G)|)}{|Z(G)|}\bigg)(-1)
+ \frac{(|G|^{3}|G'| - |Z(G)|)}{|G{}'||Z(G)|}\\
=& \frac{|G|^2(|G'| - |Z(G)|)(|G| - |G'|)}{|G'||Z(G)|},
\end{aligned}
\end{eqnarray*}
If $g\in G{}'\setminus \gamma_3(G)$ then by   \eqref{new_eq_camina_p_1} and Lemma \ref{new_camina_p-gp_1}, we have $\zeta^{w_n}_G(g) = 0$.
This completes the proof.
\end{proof}

\subsection{Groups for which $(G,Z(G))$ is a Camina pair and $(G/Z(G), Z(G/Z(G)))$ is a GCP}
In \cite{Lewis3}, Lewis began the study of those groups $G$ for
which $(G,Z(G))$ is a Camina pair and, proved that such a group $G$
must be a $p$-group for some prime $p$. Here we quote some fact about the groups for which
$(G,Z(G))$ is a Camina pair \cite[Section 4]{sunil3}:

\begin{remark}\label{remark:caminapairandGCP}
 Let $G$ be a group such that $(G,Z(G))$ is a Camina pair. Then we have the followings.
\begin{enumerate}
\item If $\chi\in \Irr(G|Z(G))$, then $\chi(G\setminus Z(G))=0$ and $\chi(1)=|G/Z(G)|^{1/2}$.
\item There is a bijection
$\Phi:\Irr(Z(G))\setminus \{1_{Z(G)}\}\longrightarrow \Irr(G|Z(G)) ~~~\textnormal{such that}$
\begin{eqnarray*}\Phi(\lambda)(g):=\left\{
\begin{array}{ll}
|G/Z(G)|^{1/2}\lambda(g) & \hbox{if $g\in Z(G)$} \\
                       0 & \hbox{otherwise.}
\end{array}
\right.
\\ \nonumber
\end{eqnarray*}
\item  $|\Irr(G|Z(G))|=|Z(G)|-1$.
\end{enumerate}
\end{remark}

Also note that if $G$ is a finite group such that $(G,Z(G))$ is a Camina pair and  $(G/Z(G), Z(G/Z(G)))$ is a generalized Camina pair then nilpotency class of $G$ is $3$.

\begin{thm}\label{thm:gcp}
Let $(G,Z(G))$ be a GCP. Then we have the following.
\begin{enumerate}
\item $cd(G)=\{1, |G:Z(G)|^{1/2} \}$.
\item The number of nonlinear irreducible characters of $G$ is $|Z(G)|-|Z(G)/G{}'|$.
\item There is a bijection
$\widehat{\Phi}: \Irr(Z(G)|G{}') \longrightarrow \nl(G)$ defined by
\begin{eqnarray*}\label{bijectiongcp}
\widehat{\Phi}(\lambda)(g):=\left\{                                                          \begin{array}{ll}
|G:Z(G)|^{1/2}\lambda(g) & \hbox{if $g\in Z(G)$,} \\
                       0 & \hbox{otherwise.}
\end{array}
\right.
\\ \nonumber
\end{eqnarray*}
\end{enumerate}
\end{thm}
\begin{proof}
Part $(1)$, $(2)$ and $(3)$  follow  from \cite[Section 3]{SR}.
\end{proof}
Note that if $(G,Z(G))$ is a Camina pair and $(G/Z(G), Z(G/Z(G)))$ is a
generalized Camina pair, then $Z(G)\subseteq G{}'$ and $(G/Z(G)){}'=G{}'/Z(G)\subseteq Z(G/Z(G))$. Now as a corollary of the above result we have the following.
\begin{cor}\label{lastsec_newcor}
Let $G$ be a finite group such that $(G,Z(G))$ is a Camina pair and  $(G/Z(G), Z(G/Z(G)))$  a generalized Camina pair.  Then we have the following.
\begin{enumerate}
\item There is a bijection
$\Psi:\Irr(Z(G/Z(G))~|~G{}'/Z(G))\longrightarrow \nl(G/Z(G)) ~~~\textnormal{such that}$
\begin{eqnarray*}\Psi(\chi)(g):=\left\{
\begin{array}{ll}
|G:Z_2(G)|^{1/2}\chi(g) & \hbox{if $g\in Z_2(G)/Z(G)$} \\
                         0 & \hbox{otherwise,}
\end{array}
\right.
\\ \nonumber
\end{eqnarray*}
where $\chi(1)=|G : Z_2(G)|^{1/2}$, $\chi\in \nl(G/Z(G))$.
\item  $|\nl(G/Z(G))|=\frac{|Z_2(G)|(|G'| - |Z(G)|)}{|G'||Z(G)|}$ and $\nl(G)=\Irr(G|Z(G))\sqcup \nl(G/Z(G))$.
\end{enumerate}
\end{cor}

\begin{prop}\label{lemma:camina&gcp}
Let $G$ be a finite group such that $(G,Z(G))$ is a Camina pair and  $(G/Z(G), Z(G/Z(G)))$  a generalized Camina pair.  Then
\begin{eqnarray*}\label{camina&generalizedcamina}
C^{w_3}(\chi) = \left\{
           \begin{array}{ll}
            \langle \zeta^{w_{2}}_G\downarrow_{Z(G)}, 1_{Z(G)}\rangle & \hbox{if $\chi\in \Irr(G|Z(G))$,} \\
            \langle \zeta^{w_{2}}_G\downarrow_{Z_{2}(G)}, 1_{Z_{2}(G)}\rangle & \hbox{if $\chi\in \nl(G/Z(G))$.}
           \end{array}
         \right.
\end{eqnarray*}
\end{prop}
\begin{proof}
By Remark \ref{remark:caminapairandGCP} one can easily compute that
$\langle \zeta^{w_{2}}_G\chi, \chi \rangle=\langle \zeta^{w_{2}}_G\downarrow_{Z(G)}, 1_{Z(G)}\rangle$, where $\chi\in \Irr(G|Z(G))$.
Again, by Corollary \ref{lastsec_newcor}, we have
$\langle \zeta^{w_{2}}_G\chi, \chi \rangle = \langle \zeta^{w_{2}}_G\downarrow_{Z_{2}(G)}, 1_{Z_{2}(G)}\rangle $, where  $\chi\in \nl(G/Z(G))$.
\end{proof}

\begin{thm}\label{prop:solutioncaminagcp}
Let $G$ be a finite group such that $(G,Z(G))$ is a Camina pair and  $(G/Z(G), Z(G/Z(G)))$  a generalized Camina pair. Then
\begin{eqnarray*}
\begin{aligned}
\zeta^{w_3}_G {} =& |G|^{2}\sum_{\chi\in \lin(G)} \chi+
\bigg(\frac{|G|^3}{|G{}'|}+ \frac{|G|^2|Z_2(G)|(|G'| - |Z_2(G)|)}{|G'||Z(G)|}\bigg) \sum_{\chi\in \Irr(G|Z(G))}\frac{\chi}{\chi(1)} \\
&+ \frac{|G|^{3}}{|Z_2(G)|}\sum_{\chi\in \nl(G/Z(G))}\frac{\chi}{\chi(1)}.
\end{aligned}
\end{eqnarray*}
\end{thm}
\begin{proof}
By Corollary \ref{lastsec_newcor} part $(2)$ we have $\nl(G)=\Irr(G|Z(G)) \sqcup ~ \nl(G/Z(G))$. Therefore, using \eqref{neweq0001}, we get
\[
\zeta^{w_3}_G = |G|^{2}\sum_{\chi\in \lin(G)} \chi+
 |G|\sum_{\chi\in \Irr(G|Z(G))}\frac{C^{w_3}(\chi)\chi}{\chi(1)}
+  |G|\sum_{\chi\in \nl(G/Z(G))}\frac{C^{w_3}(\chi)\chi}{\chi(1)}.
\]
By  Proposition \ref{lemma:camina&gcp} the above expression becomes
\begin{eqnarray}\label{eqn:camina&GCPfn1}
\begin{aligned}
\zeta^{w_3}_G {} =& |G|^{2}\sum_{\chi\in \lin(G)} \chi+
\langle \zeta^{w_{2}}_G\downarrow_{Z(G)}, 1_{Z(G)}\rangle |G| \sum_{\chi\in \Irr(G|Z(G))}\frac{\chi}{\chi(1)}  \\
&+\langle \zeta^{w_{2}}_G\downarrow_{Z_2(G)}, 1_{Z_2(G)}\rangle |G| \sum_{\chi\in \nl(G/Z(G))}\frac{\chi}{\chi(1)}.
\end{aligned}
\end{eqnarray}

Let $\phi = \sum_{\chi\in \lin(G)}\chi$. Then  $\langle \phi \downarrow_{Z_2(G)},1_{Z_2(G)}\rangle = |G|/|Z_2(G)|$, since $G'\subseteq Z_2(G)$.  Therefore
\[
\langle\zeta^{w_2}_G\downarrow_{Z_2(G)},1_{Z_2(G)}\rangle = |G|\langle \phi \downarrow_{Z_2(G)},1_{Z_2(G)}\rangle = \frac{|G|^2}{|Z_2(G)|},
\]
observing that $\langle \chi{}\downarrow_{Z_2(G)}, 1_{Z_2(G)}\rangle = 0$ for $\chi\in \Irr(G|Z(G))\sqcup ~ \nl(G/Z(G))$.


Also,  $\langle \phi \downarrow_{Z(G)},1_{Z(G)}\rangle = |G|/|G'|$, since $Z(G)\subseteq G'$ and $\phi$ is the regular character of $G/G'$. Observe that $\langle\chi{}\downarrow_{Z(G)}, 1_{Z(G)}\rangle = 0$ for $\chi\in \Irr(G|Z(G))$ and
$\langle \chi{}\downarrow_{Z(G)}, 1_{Z(G)}\rangle = \chi(1)$ for $\chi\in \nl(G/Z(G))$.Therefore,
\begin{eqnarray*}
\begin{aligned}
\langle\zeta^{w_2}_G\downarrow_{Z(G)},1_{Z(G)}\rangle {} =
& \frac{|G|^2}{|G'|}+|G||\nl(G/Z(G))| = \frac{|G|^2}{|G{}'|}+ \frac{|G||Z_2(G)|(|G'| - |Z_2(G)|)}{|G'||Z(G)|},
\end{aligned}
\end{eqnarray*}
using Corollary \ref{lastsec_newcor} part (2). Now putting the values of $\langle\zeta^{w_2}_G\downarrow_{Z(G)},1_{Z(G)}\rangle$ and $\langle\zeta^{w_2}_G\downarrow_{Z(G)},1_{Z(G)}\rangle$ in \eqref{eqn:camina&GCPfn1} we get the required expression for $\zeta^{w_3}_G $.
\end{proof}

The following lemma is useful in proving the next result.
\begin{lem}\label{lastlemma}
Let $G$ be a finite group such that $(G,Z(G))$ is a Camina pair and  $(G/Z(G), Z(G/Z(G)))$  a generalized Camina pair.  Then
\[
\sum_{\chi\in \Irr(G|Z(G))}\frac{\chi(g)}{\chi(1)} = \begin{cases}
-1 &\text{ if } 1 \neq g \in Z(G)\subseteq G'\\
\;\;\, 0  &\text{ if } g \in G' \setminus Z(G)
\end{cases}
\]
and
\[
\sum_{\chi\in \nl(G/Z(G))}\frac{\chi(g)}{\chi(1)} = \begin{cases}
\frac{|Z_2(G)|(|G'| - |Z(G)|)}{|G'||Z(G)|} &\text{ if } 1 \neq g \in Z(G)\subseteq G'\\
 -\frac{|Z_2(G)|}{|G')|}  &\text{ if } g \in G' \setminus Z(G).
\end{cases}
\]
\end{lem}
\begin{proof}
Follows from Remark \ref{remark:caminapairandGCP} part (1), (2) and Corollary \ref{lastsec_newcor}.
\end{proof}
We conclude this subsection by the following result.
\begin{thm}\label{prop:solutioncaminagcp_1}
Suppose $(G,Z(G))$ is a Camina pair and $(G/Z(G), Z(G/Z(G)))$ a generalized Camina pair.  If $g \in G'$ then
\[
\zeta^{w_{2}}_G(g) = \begin{cases}
\frac{|G|^2[|Z(G)|(|G| + |G'|(|Z(G)| - 1)) + |Z_2(G)|(|G'| - |Z(G)|)]}{|G'||Z(G)|} &\text{ if } g=1 \\
\frac{|G|[|Z(G)|(|G| - |G'|) + |Z_2(G)|(|G'| - |Z(G)|)]}{|G'||Z(G)|} &\text{ if } g\in Z(G)\subseteq G{}'\\
\frac{|G|(|G| - |Z_2(G)|)}{|G'|}&\text{ if } g\in G{}'\setminus Z(G)
\end{cases}
\]
and
\[
\zeta^{w_{3}}_G(g) = \begin{cases}
\frac{|G|^2[|G||Z(G)|^2 + (|G'| - |Z(G)|)(|G| + |Z_2(G)|(|Z(G)| - 1))]}{|G'||Z(G)|}&\text{ if } g=1 \\
\frac{|G|^2(|G'| - |Z(G)|)(|G| - |Z_2(G)|)}{|G'||Z(G)|} &\text{ if } g\in Z(G)\subseteq G{}'\\
0 &\text{ if } g\in G{}'\setminus Z(G).
\end{cases}
\]
\end{thm}

\begin{proof}
If $g = 1$, then by \eqref{Frobeq} we have
\[
\zeta^{w_2}_G(1)=|G||\Irr(G)| = \frac{|G|^2[|Z(G)|(|G| + |G'|(|Z(G)| - 1)) + |Z_2(G)|(|G'| - |Z(G)|)]}{|G'||Z(G)|},
\]
noting that $\Irr(G) = \lin(G) \sqcup \Irr(G|Z(G)) \sqcup ~ \nl(G/Z(G))$. If $1\neq g\in  G'$ then by \eqref{Frobeq} we have
\begin{equation}\label{000000}
\zeta^{w_2}_G(g) = \frac{|G|^2}{|G'|} + |G|\sum_{\chi\in \Irr(G|Z(G))}\frac{\chi(g)}{\chi(1)} + |G|\sum_{\chi\in \nl(G/Z(G))}\frac{\chi(g)}{\chi(1)}
\end{equation}
Now, if $1\neq g\in Z(G)\subseteq G'$, then by \eqref{000000} and Lemma \ref{lastlemma} we have
\begin{align*}
\zeta^{w_2}_G(g) = & \frac{|G|^2}{|G'|} + |G|(-1)  + \frac{|G||Z_2(G)|(|G'| - |Z(G)|)}{|G'||Z(G)|}\\
= & \frac{|G|[|Z(G)|(|G| - |G'|) + |Z_2(G)|(|G'| - |Z(G)|)]}{|G'||Z(G)|}.
\end{align*}
If $g\in G'\setminus Z(G)$, then by \eqref{000000} and Lemma \ref{lastlemma} we have
\begin{align*}
\zeta^{w_2}_G(g) = & \frac{|G|^2}{|G'|} +   \frac{-|G||Z_2(G)|}{|G'|} = \frac{|G|(|G| - |Z_2(G)|)}{|G'|}.
\end{align*}

For the second part, if $g = 1$, then by Theorem \ref{prop:solutioncaminagcp}, we have
\begin{eqnarray*}
\begin{aligned}
\zeta^{w_3}_G(1)= & \frac{|G|^{3}}{|G'|} + \bigg(\frac{|G|^3}{|G{}'|}+ \frac{|G|^2|Z_2(G)|(|G'| - |Z_2(G)|)}{|G'||Z(G)|}\bigg) |\Irr(G|Z(G))| + \frac{|G|^{3}|\nl(G/Z(G))|}{|Z_2(G)|}\\
=& \frac{|G|^2[|G||Z(G)|^2 + (|G'| - |Z(G)|)(|G| + |Z_2(G)|(|Z(G)| - 1))]}{|G'||Z(G)|},
\end{aligned}
\end{eqnarray*}
using Remark \ref{remark:caminapairandGCP} (part 3) and Corollary \ref {lastsec_newcor} (part 2).

If $1\neq g\in G'$, then by Theorem \ref {prop:solutioncaminagcp} we have
\begin{eqnarray}\label{lasteq}
\begin{aligned}
\zeta^{w_3}_G (g) =& \frac{|G|^{3}}{|G'|}+
\bigg(\frac{|G|^3}{|G{}'|}+ \frac{|G|^2|Z_2(G)|(|G'| - |Z_2(G)|)}{|G'||Z(G)|}\bigg) \sum_{\chi\in \Irr(G|Z(G))}\frac{\chi(g)}{\chi(1)} \\
&+ \frac{|G|^{3}}{|Z_2(G)|}\sum_{\chi\in \nl(G/Z(G))}\frac{\chi(g)}{\chi(1)}.
\end{aligned}
\end{eqnarray}
If $1\neq g\in Z(G)\subseteq G'$, then by \eqref{lasteq} and Lemma \ref{lastlemma} we have
\begin{eqnarray*}
\begin{aligned}
\zeta^{w_3}_G (g) =& \frac{|G|^{3}}{|G'|}+
\bigg(\frac{|G|^3}{|G{}'|}+ \frac{|G|^2|Z_2(G)|(|G'| - |Z_2(G)|)}{|G'||Z(G)|}\bigg) (-1)
+  \frac{|G|^{3}(|G'| - |Z(G)|)}{|G'||Z(G)|}\\
=& \frac{|G|^2(|G'| - |Z(G)|)(|G| - |Z_2(G)|)}{|G'||Z(G)|}.
\end{aligned}
\end{eqnarray*}

\noindent  If $g\in G'\setminus Z(G)$, then by \eqref{lasteq} and Lemma \ref{lastlemma}, we have
$\zeta^{w_3}_G (g) = 0$. This completes the proof.
\end{proof}


\subsection{Groups with unique nonlinear irreducible character}

The class of groups considered in this subsection is a subclass of the groups $G$ with $|cd(G)| = 2$, where $cd(G) = \{\chi(1) : \chi \in \Irr(G)\}$. Groups having $|cd(G)| = 2$ are considered in \cite{mP04} for the word $w_2(x_1, x_2)=[x_1,x_2]$. Note that explicit formula for $\zeta^{w_2}_G(g)$ can be obtained from  \cite[Theorem 2.2]{mP04} for these groups. In this subsection, we derive  explicit formula for $\zeta^{w_n}_G(g)$ where $n \geq 3$ for groups having unique nonlinear irreducible character. We start with the following theorem which is a consequence of the main theorem in \cite{BCH}.
\begin{thm}\label{thm:distinctdegree}
Let $G$ be a non-abelian finite group having unique nonlinear irreducible character. Then one of the following holds:
\begin{enumerate}
\item $G$ is an extra-special $2$-group.
\item  $G$ is a Frobenius group of order $p^m(p^m-1)$ for some prime power $p^m$ with
an abelian Frobenius kernel of order $p^m$, a cyclic Frobenius complement and $Z(G)=\{1\}$. Here Frobenius kernel is the commutator subgroup of $G$
and $G$ acts transitively, by conjugation, on $G{}'\setminus \{1\}$.
\end{enumerate}
\end{thm}
\noindent Note that the  groups in the above theorem are Camina groups. Further, group of type $(1)$ has nilpotency class $2$ whereas group of type $(2)$ is not nilpotent. Therefore,  we derive  explicit formula of $\zeta^{w_n}_G$, where $n \geq 3$, only for the group $G$ of type ~$(2)$.
We need the following lemma.
\begin{lem}\label{lemma:uniquenon}
 Let $G$ be a group with unique nonlinear irreducible character $\chi$ with $Z(G)=\{1\}$.
 Then we have the following.
\begin{enumerate}
\item $\chi$ vanishes outside $G{}'$  and $\chi(1)=p^{m}-1$.
\item $\chi^2=\sum_{\phi\in \lin(G)}\phi+ (p^m-2)\chi$.
\item $\chi(g)=-1$ for all $g\in G{}'\setminus \{1\}$.
\end{enumerate}
\end{lem}
\begin{proof} By Theorem \ref{thm:distinctdegree}, $G$ is a Frobenius group of order $p^m(p^m-1)$ for some prime power $p^m$.\\
\noindent $(1)$ Since $G$ is a Camina group,   $\chi$ vanishes outside $G{}'$. The second part follows from the relation
$|G|=\sum_{\phi\in \Irr(G)}\phi(1)^2$.

%
\noindent $(2)$ Since $\chi^2$ is a character of $G$, $\chi^2=\sum_{\phi\in \Irr(G)}c_{\phi}\phi$,
where $c_{\phi}=\langle \chi^2,\phi\rangle \in \mathbb{N}\cup \{0\}$. Let $\phi\in \lin(G)$. Then
$c_{\phi}=\langle \chi^2, \phi\rangle=\langle \chi, \overline{\chi}\phi\rangle=\langle \chi, \chi\phi\rangle=\langle \chi, \chi\rangle=1$.
Therefore,
$$
\chi^2=\sum_{\phi\in \lin(G)}\phi+ c_{\chi}\chi.
$$
 Now, by degree computation
of both the sides of the above expression, we get $c_{\chi}=p^m-2$. Hence the result follows.

\noindent $(3)$ Follows from the fact that $\langle \chi, \lambda \rangle=0$ for $\lambda\in \lin(G)$.
\end{proof}

\begin{thm}\label{prop:gencommutator}
Let $G$ be a group with unique nonlinear irreducible character $\phi$ and $Z(G)=\{1\}$.
If $n \geq 3$, then
\[
C^{w_n}(\phi) = \frac{|G|^{n-1}}{|G{}'|}+ \frac{|G|}{(p^m-1)}C^{w_{n-1}}(\phi)(p^m-2)
\]
and hence $\zeta^{w_{n}}_G = |G|^{n-1}\sum_{\chi\in \lin(G)}~\chi+\frac{|G|}{(p^m-1)}C^{w_{n}}(\phi)\phi$.
\end{thm}
\begin{proof} By Theorem \ref{thm:distinctdegree}, $G$ is a Frobenius group of order $p^m(p^m-1)$ for some prime power $p^m$.
We shall prove the proposition by induction on $n$. By Lemma \ref{lemma:uniquenon} (2), we have
\begin{eqnarray*}
C^{w_3}(\phi) =\bigg\langle \bigg(\sum_{\chi\in \Irr(G)}\frac{|G|}{\chi(1)}\chi\bigg) \phi,\phi \bigg\rangle
= \sum_{\chi\in \Irr(G)}\frac{|G|}{\chi(1)}\big\langle \chi\phi,\phi \big\rangle
=\frac{|G|^2}{|G{}'|}+ \frac{|G|}{(p^m-1)}C^{w_2}(\phi)(p^m-2).
\end{eqnarray*}
Hence, by \eqref{neweq0001}, we have
\begin{equation*}\label{eqn:w3}
\zeta^{w_3}_G = |G|^2\sum_{\chi\in \lin(G)}~\chi + \frac{|G|}{(p^m-1)}  C^{w_3}(\phi)  \phi,
\end{equation*}
Thus  induction begins at $n = 3$.
Suppose that the proposition is true for $n-1$. Then
\[
C^{w_{n-1}}(\phi)= \frac{|G|^{n-2}}{|G{}'|}+ \frac{|G|}{(p^m-1)}C^{w_{n-2}}(\phi)(p^m-2)
\]
and hence
\[
\zeta^{w_{n-1}}_G = |G|^{n-2}\sum_{\chi\in \lin(G)}~\chi+\frac{|G|}{(p^m-1)}C^{w_{n-1}}(\phi)\phi.
\]
Therefore
\begin{eqnarray*}
C^{w_{n}}(\phi) &=& \langle \zeta^{w_{n-1}}_G \phi,\phi \rangle\\
&=&\bigg\langle \bigg(\sum_{\psi\in \lin(G)}\frac{|G|^{n-2}}{\psi(1)}\psi+\frac{|G|}{(p^m-1)}C^{w_{n-1}}(\phi)\phi \bigg)~\phi, \phi\bigg\rangle \\
&=& \sum_{\psi\in \lin(G)}\frac{|G|^{n-2}}{\psi(1)}\langle \phi,\phi\rangle +\frac{|G|}{(p^m-1)}C^{w_{n-1}}(\phi)\langle
 \phi^2,\phi \rangle\\
&=&\frac{|G|^{n-1}}{|G{}'|}+ \frac{|G|}{(p^m-1)}C^{w_{n-1}}(\phi)(p^m-2)~~~~~~~~~~~~~~~(\textnormal{using Lemma \ref{lemma:uniquenon} (2)}).
\end{eqnarray*}
Now,  by \eqref{neweq0001}, we have
\[
\zeta^{w_{n}}_G = |G|^{n-1}\sum_{\chi\in \lin(G)}~\chi+\frac{|G|}{(p^m-1)}C^{w_{n}}(\phi)\phi.
\]
This proves  the theorem.
\end{proof}


\begin{thm}\label{Appl_Prop_1}
Let $G$ be a group with unique non-linear irreducible character $\phi$ with $Z(G)=\{1\}$. If $g \in G'$ and $n \geq 3$ then
\[
\zeta^{w_{n}}_G(g) = \begin{cases}
\frac{|G|^n}{|G'|} + |G|C^{w_{n}}(\phi) &\text{ if } g = 1 \\
\frac{|G|^n}{|G'|} -\frac{|G|}{(p^m-1)}C^{w_{n}}(\phi) &\text{ if } g \neq 1.
\end{cases} 
\]

In particular, for $n=3$
\[
\zeta^{w_{3}}_G(g) = \begin{cases}
\frac{2|G|^3}{|G'|} + \frac{|G|^2}{p^m-1}(p^m-2) &\text{ if } g = 1 \\
\frac{|G|^3}{|G'|} -\frac{|G|^2}{(p^m-2)}(\frac{|G|}{|G{}'|}+\frac{p^m-2}{p^m-1}) &\text{ if } g \neq 1.
\end{cases} 
\]
\end{thm}
\begin{proof} The proof follows from Theorem \ref{prop:gencommutator} and Lemma \ref{lemma:uniquenon} (3).
%
\end{proof}
It is easy to see that the nonlinear irreducible character
comes as a induction of a linear character of the abelian commutator subgroup
for the group of type (2) in Theorem \ref{thm:distinctdegree}. So it is natural to ask for some bounds of $C^{w_n}(\chi)$ if all the nonlinear irreducible characters
comes as a induction of some linear character of a fixed abelian normal subgroup $N$ of $G$. 
We conclude this paper by the following upper bound of $C^{w_n}(\chi)$.
\begin{prop}\label{prop:cd2}
Let $cd(G)=\{1,m\}$ and let $G$ has an abelian normal subgroup of index $m$ such that every nonlinear irreducible character of $G$ is an induction of some
irreducible character of $N$. Then for $n\geq 3$, $C^{w_n}(\chi)   \leq m\frac{|G|^{n-1}}{|N|}$ whenever $\chi\in \nl(G)$.
\end{prop}
\begin{proof}
Let $N$ be a normal subgroup of index $m$. Then $G/N$
has no nonlinear irreducible character and hence $G/N$ is abelian. Therefore $G{}'\subseteq N$.

 Let $\chi\in \nl(G)$ then $\psi{}\uparrow^{G}_N =\chi$
for some $\psi\in \Irr(N)$. Therefore $\chi(G\setminus N)=0$.
Thus

\begin{eqnarray*}
C^{w_n}(\chi)=\langle \zeta^{w_{n-1}}_G\chi, \chi\rangle &=&\frac{1}{|G|}\sum_{g\in N}\zeta^{w_{n-1}}_G(g)|\chi(g)|^2\\
&\leq & \frac{\chi(1)^2}{|G|}\sum_{g\in N}\zeta^{w_{n-1}}_G(g)\\
&=&\frac{m}{|N|}\sum_{g\in N}\zeta^{w_{n-1}}_G(g)=\chi(1)\langle \zeta^{w_{n-1}}_G{}\downarrow_N, 1_N\rangle.
\end{eqnarray*}
Let $\varphi = \sum_{\psi\in \lin(G)}\psi$. Using expression of $\zeta^{w_{n-1}}_G$
 and Proposition \ref{GCPcoffeicientofirr_prop1}, observing that $\langle \chi{}\downarrow_N, 1_N\rangle=0$, we get

\begin{eqnarray*}
\langle \zeta^{w_{n-1}}_G{}\downarrow_N, 1_N\rangle&=& |G|^{n-2} \langle \varphi\downarrow_{N}, 1_N\rangle
+\sum_{\phi\in \nl(G)} \frac{|G|}{\phi(1)}\langle  \zeta^{w_{n-2}}_G\phi,\phi\rangle \langle \phi{}\downarrow_N,1_N\rangle\\
&=& |G|^{n-2} \frac{|G|}{|N|}= \frac{|G|^{n-1}}{|N|}\,\,\, ~~~\textnormal{(since $G{}'\subseteq N$)}.
\end{eqnarray*}
Therefore, $C^{w_n}(\chi)\leq m\frac{|G|^{n-1}}{|N|}$ for $\chi\in \nl(G)$.
\end{proof}

\section{Acknowledgements}
The first author is supported by Lady Davis fellowship and thankful to Aner Shalev for his support
to this research work.


\begin{thebibliography}{30}
\bibitem{abert06}
M. Ab$\acute{\rm e}$rt, On the probability of satisfying a word in a group {\em J. Group Theory} {\bf 9}(5) (2006),  685--694.

\bibitem{aV11}
A. Amit and U. Vishne, { Characters and solutions to equations in finite  groups}, {\em J. Algebra  Appl.} {\bf 10}(4)  (2011), 675--686.

\bibitem{BCH}
Y. Berkovich, D. Chillag and M. Herzog, Finite groups in which the degrees of the nonlinear irreducible characters are distinct,
{\em Proc. Amer. Math. Soc.}, {\bf 115} (1992), 955-958.


\bibitem{Dark}
R. Dark and C. M. Scoppola, On Camina groups of prime power order, {\em J.
Algebra}, {\bf 181} (1996) 787--802.


\bibitem{DN09}
A. K. Das and  R. K. Nath,  On solutions of a class of equations in a finite group, {\em Comm. Algebra} {\bf 37}(11) (2009), 3904--3911.

\bibitem{DN10}
A. K. Das and  R. K. Nath,   On generalized relative commutativity degree of a finite group,
{\em Int. Electron. J. Algebra} {\bf 7} (2010), 140--151.

\bibitem{DN11}
A. K. Das and  R. K. Nath,  A generalization of commutativity degree of finite groups, {\em Comm. Algebra} {\bf 40}(6) (2011), 1974--1981.


\bibitem{sG07}
S. Garion and A. Shalev, { Commutator maps, measure preservation, and T-systems}, {\em Trans. Amer. Math. Soc.} \textbf{361}(9) (2009) 4631--4651.


\bibitem{Hall}
P. Hall, Verbal and marginal subgroups, {\em  J. reine und ang. Math.}, {\bf 182} (1940), 156-157.




\bibitem{fro68}
Frobenius, F. G. (1968) { \it $\Ddot{U}$ber Gruppencharaktere}, Gesammelte
Abhandlungen Band III, p. 1--37 (J. P. Serre, ed.), Springer-Verlag, Berlin.


\bibitem{Isaacs}
I. M. Isaacs,{\em Character theory of finite groups},
AMS Chelsea Publishing, Academic Press, New York, 2000.

\bibitem{Lewis4}
M. L. Lewis, The vanishing-off subgroup, {\em J. Algebra}, {\bf 321}(4) (2009) 1313--1325.

\bibitem{Lewis1}
M. L. Lewis, {Character tables of groups where all nonlinear irreducible characters
vanish off the center}, in {\em Proc. Ischia Group Theory}, 2008.

\bibitem{Lewis2}
 M. L. Lewis, Generalizing Camina groups and their character tables, {\em J. Group Theory}, {\bf 12} (2009) 209–-218.

\bibitem{Lewis3}
M. L. Lewis, On $p$-group Camina pairs, {\em J. Group Theory}, {\bf 15} (2012) 469--483.





\bibitem{mrM05}
 M. R. R. Moghaddam, A. R. Salemkar, and K. Chiti, $n$-Isoclinism classes
and $n$-nilpotency degree of finite groups, {\em Algebra Colloq.} {\bf 12}(2) (2005), 225--261.

\bibitem{nath14}
R. K. Nath, A new class of almost measure preserving maps on finite simple groups {\em J. Algebra Appl.} {\bf 13}(4) (2014), 1350142, 5 pp.

\bibitem{nath11}
R. K. Nath and A. K. Das, On generalized commutativity degree of a finite group {\em Rocky Mountain J. Math.} {\bf 41}(6) (2011), 1987--2000.


\bibitem{nY2015}
R. K. Nath and M. K. Yadav, On the probability distribution associated to commutator word map in  finite groups, {\em Int. J. Algebra Compt.} {\bf 25}(7) (2015), 1107--1124.

\bibitem{nikolov07}
 N. Nikolov and   D. Segal, A characterization of finite soluble groups, {\em Bull. Lond. Math. Soc.} {\bf 39}(2) (2007),  209--213.




\bibitem{pS12}
O. Parzanchevski and G. Schul, On the Fourier expansion of word maps, {\em Bull. Lond. Math. Soc.} {\bf 46}(1) (2014),   91--102.

\bibitem{mP04}
M. R. Pournaki and  R. Sobhani, {Probability that the commutator of two group elements is equal to a given element}, {\em J. Pure Appl. Alg.} {\bf 212} (2008), 727--734.

\bibitem{sunil3}
 S. K. Prajapati and B. Sury, {On the total Character of finite groups}, {\em Int. J. Group Theory}, 3(3) (2014) 47-67.

\bibitem{SR}
 S. K. Prajapati and R. Sarma, Total Character of a group $G$ with $(G,Z(G))$ as a
generalized Camina pair, {\em Canad. Math. Bull.}, {\bf 59} (2) (2016), 392-402.

\bibitem{SR15}
S. K. Prajapati and R. Sarma,  On group equations, {\em Bull. Iranian Math. Soc.} {\bf 41}(2) (2015), 315–324.

\bibitem{SR13}
S. K. Prajapati and R. Sarma, On the solutions of $x^k = g$ in a finite group, {\em Bull. Korean Math. Soc}. {\bf 50}(2) (2013), 697--704.

\bibitem{Sol69}
L. Solomon,  The solutions of equations in groups. {\em Arch. Math.} {\bf 20} (1969), 241--247.



\bibitem{sP95}
S. P. Strunkov,  On the theory of equations in finite groups,  {\em Izv. Math.} {\bf 59} (6) (1995), 1273--1282.


\bibitem{tT98}
 Tambour, T.   { The number of solutions of some equations in finite
groups and a new proof of It$\hat{o}$'s theorem}, {\em Comm. Algebra},
{\bf 28}(11) (2000), 5353--5361.


























\end{thebibliography}
\end{document}